\documentclass{article}
\usepackage{amsfonts}
\usepackage{amsmath}
\usepackage{amssymb}
\usepackage{amsthm}
\usepackage[margin=1.25in]{geometry}
\usepackage{hyperref}
\usepackage{tikz}
\usepackage{graphicx}
\usepackage{placeins}

\usetikzlibrary{decorations.markings}
\tikzstyle{vertex}=[circle, draw, fill=black, inner sep=0pt, minimum size=6pt]
\newcommand{\vertex}{\node[vertex]}

\newtheorem{theorem}{Theorem}[section]
\newtheorem{lemma}[theorem]{Lemma}

\newtheorem{proposition}[theorem]{Proposition}
\newtheorem{corollary}[theorem]{Corollary}
\newtheorem{observation}{Observation}
\newtheorem{conjecture}{Conjecture}

\newcommand{\RR}{\mathbb{R}}
\newcommand{\Sn}{\mathcal{S}_n}
\newcommand{\Sk}{\mathcal{S}_k}
\newcommand{\red}{\text{red}}

\begin{document}

\title{Competition graphs induced by permutations}

\author{Brian Nakamura\thanks{CCICADA/DIMACS, Rutgers University-New Brunswick, Piscataway, NJ, USA. [bnaka@dimacs.rutgers.edu]} \; and Elizabeth Yang\thanks{Princeton University, Princeton, NJ, USA. [eyang@princeton.edu]}}

\date{}

\maketitle

\begin{abstract}
	In prior work, Cho and Kim studied competition graphs arising from doubly partial orders. In this article, we consider a related problem where competition graphs are instead induced by permutations. We first show that this approach produces the same class of competition graphs as the doubly partial order. In addition, we observe that the $123$ and $132$ patterns in a permutation induce the edges in the associated competition graph. We classify the competition graphs arising from $132$-avoiding permutations and show that those graphs must avoid an induced path graph of length $3$. Finally, we consider the weighted competition graph of permutations and give some initial enumerative and structural results in that setting.
\end{abstract}

\medskip

\section{Introduction}\label{intro}

Given a digraph $D = (V,A)$, the \emph{competition graph} $G = C(D)$ of $D$ is the undirected graph that has the same vertex set as $D$ and has edge $xy$ if and only if there exists a vertex $u \in V$ such that both the arcs $(x,u)$ and $(y,u)$ are in $D$. Competition graphs were first introduced by Cohen \cite{cohen} as a way to study food webs in ecology, where vertices represented different species in an ecosystem and a directed edge (in the digraph) existed from species A to species B if A \emph{preyed} on B. In this context, competition graphs are undirected graphs where an edge exists if two species feed on the same prey (i.e., they are in competition for resources). 

One active research effort has been to study competition graphs arising from interesting families of digraphs. For example, competition graphs of acyclic digraphs were studied in~\cite{duttbrig,robste} while those of Hamiltonian digraphs were studied in~\cite{FLMMP,guichard}. More recently, Cho and Kim \cite{ChoKim} studied competition graphs arising from doubly partial orders, which was a problem posed to them by Roberts as a means of extending results in~\cite{kimrob}. 

Let $S$ be a finite subset of $\RR^{2}$. Cho and Kim define a relation $\prec$ as a \emph{doubly partial order} on $S$ if $(x,y) \prec (z,w)$ whenever $x < z$ and $y < w$ for $(x,y), (z,w) \in S$. The set $S$ induces a digraph $D = (V,A)$ (under this relation) by letting the points in $S$ become the vertices of $D$ and arc $(u,v) \in A$ if and only if $v \prec u$ (in the doubly partial order on $S$). We will refer to $D$ as a doubly partial order. Cho and Kim \cite{ChoKim} showed that the competition graphs of these doubly partial orders are interval graphs, which are intersection graphs of a set of intervals on the real line. They also showed that any interval graph, with sufficiently many isolated vertices, is the competition graph of some doubly partial order.

In this article, we will consider competition graphs induced by certain families of this doubly partial order and will show why this is a natural consideration. Let $\pi = \pi_{1} \ldots \pi_{n} \in \Sn$ be a permutation of length $n$ in one-line notation. Given a finite sequence $s_{1} s_{2} \ldots s_{n}$ of distinct real numbers, we define the \emph{reduction} of this sequence, denoted by $\red(s_{1} \ldots s_{n})$, to be the permutation $\pi_{1} \ldots \pi_{n} \in \Sn$ that is order-isomorphic to the sequence (that is, $\pi_{i} < \pi_{j}$ if and only if $s_{i} < s_{j}$ for every $1 \leq i,j \leq n$).

We say that the permutation $\pi \in \Sn$ \emph{contains} the (permutation) pattern $\tau \in \Sk$ if there exists some $1 \leq i_{1} < i_{2} < \ldots < i_{k} \leq n$ such that $\red(\pi_{i_{1}} \pi_{i_{2}} \ldots \pi_{i_{k}}) = \tau$ (i.e., the subsequence $\pi_{i_{1}} \ldots \pi_{i_{k}}$ and $\tau$ are order-isomorphic).  Such a subsequence will be referred to as an \emph{occurrence} of $\tau$. We say that permutation $\pi$ \emph{avoids} the pattern $\tau$ if $\pi$ does not contain $\tau$. For example, the permutation $\pi = 53412$ avoids the pattern $\tau = 123$, while the permutation $\pi^{\prime} = 52134$ contains two occurrences of $\tau$ (given by the $2 3 4$ and $1 3 4$ subsequences). The set of length $n$ permutations avoiding the pattern $\tau$ is denoted by $\Sn(\tau)$. Additionally, the number of length $n$ permutations avoiding $\tau$ is denoted by $s_{n}(\tau) := | \Sn(\tau) |$. The patterns $\sigma$ and $\tau$ are said to be \emph{Wilf-equivalent} if $s_{n}(\sigma) = s_{n}(\tau)$ for every $n$.

The study of patterns in permutations gained interest after Knuth used the notion to describe permutations that are stack-sortable \cite{knuth}. For patterns of length $2$, it is (trivially) known that $s_{n}(\tau) = 1$ for all $n$. For each pattern $\tau$ of length $3$, it is known that $s_{n}(\tau) = \frac{1}{n+1}{2 n \choose n}$ (the Catalan numbers). In this case, the patterns $1 2 3$ and $3 2 1$ are ``trivially'' Wilf-equivalent, while the patterns $1 3 2$, $2 1 3$, $2 3 1$, and $3 1 2$ are trivially Wilf-equivalent to one another. The patterns $1 2 3$ and $1 3 2$ are Wilf-equivalent for non-trivial reasons. For further background, the reader is directed to \cite{kitaev:book, kitman:survey, stein:survey}.

There is a natural connection between permutations and the doubly partial order described earlier. A permutation $\pi = \pi_{1} \ldots \pi_{n}$ is often visualized as the $n$ points $\{ (i, \pi_{i}) \}_{1 \leq i \leq n}$ in a grid or in $\RR^{2}$. For example, the permutation $461532$ is shown in Figure~\ref{fig461532}.\\

\begin{figure}[h!]
	\[\begin{tikzpicture}[scale=0.75]
		\node (p1) at (0.5,3.5) [circle,fill=black] {};
		\node (p2) at (1.5,5.5) [circle,fill=black] {};
		\node (p3) at (2.5,0.5) [circle,fill=black] {};
		\node (p4) at (3.5,4.5) [circle,fill=black] {};
		\node (p5) at (4.5,2.5) [circle,fill=black] {};
		\node (p6) at (5.5,1.5) [circle,fill=black] {};
	
		\draw[step=1.0,black,thin] (0,0) grid (6,6);
	
		\node (lab1) at (0.5,0) [label=below:$4$] {};
		\node (lab2) at (1.5,0) [label=below:$6$] {};
		\node (lab3) at (2.5,0) [label=below:$1$] {};
		\node (lab4) at (3.5,0) [label=below:$5$] {};
		\node (lab5) at (4.5,0) [label=below:$3$] {};
		\node (lab6) at (5.5,0) [label=below:$2$] {};
	
	\end{tikzpicture}\]
	\caption{Visualization of permutation $461532$}\label{fig461532}
\end{figure}
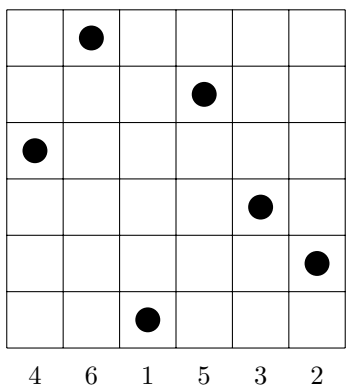

We refer to the doubly partial order on these points as the \emph{doubly partial order on permutation} $\pi$. We will say that a digraph is induced from permutation $\pi$ when it is the induced digraph from the doubly partial order on $\pi$. This will be denoted by $D(\pi)$. Analogously, the competition graph induced by permutation $\pi$ is the competition graph of the digraph $D(\pi)$. This will be denoted by $C(D(\pi))$ or more simply $C(\pi)$. The digraph and competition graph of permutation $461532$ is shown in Figure~\ref{fig461532CG}.\\

\begin{figure}[h!]
	\centering
	\includegraphics[width=\textwidth]{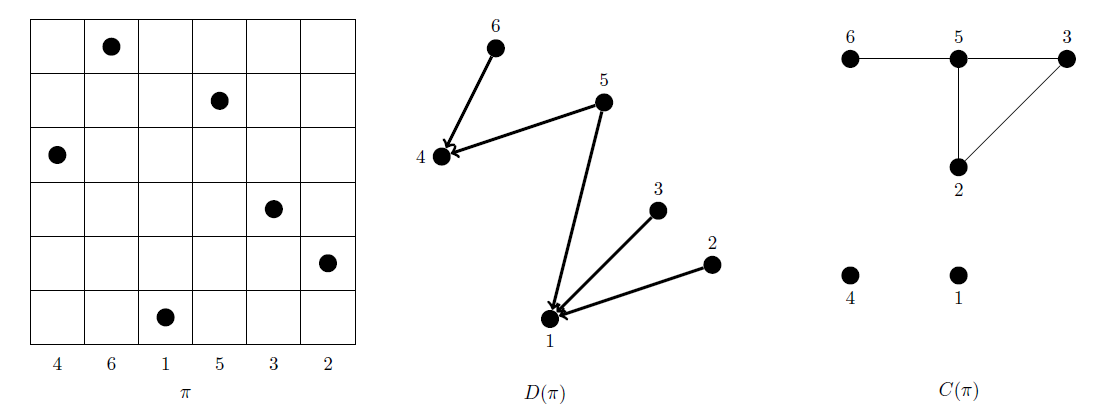}
	\caption{Digraph and competition graph of $\pi = 461532$}\label{fig461532CG}
\end{figure}

In this article, we will present results regarding competition graphs for various permutation classes. In Section~\ref{prelim}, we will present some observations and results that motivate the restriction of general finite subsets of $\RR^{2}$ to sets of points arising from permutations. In Section~\ref{compgraphs}, we will provide a characterization of competition graphs arising from permutations in $\Sn(132)$ and some partial results on competition graphs arising from $\Sn(123)$. In Section~\ref{wcompgraphs}, we will define the notion of \emph{weighted competition graphs} and present enumerative results on such graphs that have certain structures. We end with some suggestions for future work in Section~\ref{concl}.

\FloatBarrier
\section{Preliminary observations and motivation}\label{prelim}

Given a permutation $\pi \in \Sn$, recall that $D(\pi)$ is the digraph induced by the doubly partial order on $\pi$ and that $C(\pi)$ is the competition graph of $D(\pi)$. For a set of permutations $T \subseteq \Sn$, we write $C(T)$ to denote the set of graphs $\{ C(\pi) : \pi \in T \}$. For example, $C(\Sn)$ is the set of all competition graphs arising from a length $n$ permutation.

Given a permutation $\pi = \pi_{1} \ldots \pi_{n}$, the competition graph $C(\pi)$ will have vertices corresponding to the points $(1, \pi_{1})$, $(2, \pi_{2})$, and so on. In general, we will refer to the vertex arising from $(i,\pi_{i})$ as simply $\pi_{i}$. When two points satisfy the doubly partial order, say $(j,\pi_{j}) \succ (i,\pi_{i})$, we will often refer to $\pi_{j}$ as the \emph{predator} (vertex) and to $\pi_{i}$ as the \emph{prey} (vertex). Equivalently, we may also say that $\pi_{j}$ \emph{preys on} $\pi_{i}$.

We now give some observations and results to motivate this new direction.

\begin{proposition}
	Let $S \subseteq \RR^{2}$ with $|S| = n$. The competition graph of the doubly partial order on $S$ is isomorphic to $C(\pi)$ for some $\pi \in \Sn$.
\end{proposition}

\begin{proof}
	Let $S \subseteq \RR^{2}$ with $|S| = n$ and let $G$ be the competition graph of the doubly partial order on $S$. Without loss of generality, suppose $S = \{ (x_{1}, y_{1}), (x_{2}, y_{2}), \ldots, (x_{n},y_{n}) \}$ such that $x_{1} \leq x_{2} \leq \ldots \leq x_{n}$ and if $x_{i} = x_{j}$ (for some $i < j$) then $y_{i} < y_{j}$. If no two points in $S$ share the same $x$ or $y$ coordinates, then the permutation $\pi = \red(y_{1} y_{2} \ldots y_{n})$ will induce a competition graph that is isomorphic to $G$.
	
	Suppose that some of the points share a common $x$ coordinate. Let $(x_{k},y_{k}), (x_{k+1}, y_{k+1}), \ldots, (x_{l},y_{l})$ be the first maximal subset of points having the same $x$ coordinate ($x_{k} = x_{k+1} = \ldots = x_{l}$). Recall that $y_{k} < y_{k+1} < \ldots < y_{l}$. Let $d = x_{k}-x_{k-1}$ (if $k=1$, let $d = 1$), and let $\delta = \frac{d}{l - k + 1}$. We ``shift'' $(x_{k},y_{k}), \ldots, (x_{l},y_{l})$ to the left so that they lie on a line with negative slope:
\begin{align*}
	S^{\prime} = (S \backslash \{ (x_{k},y_{k}), (x_{k+1}, y_{k+1}), \ldots, (x_{l},y_{l}) \}) \cup \{ (x_{k + i} - i \delta,y_{k + i})\}_{0 \leq i \leq (l-k)}.
\end{align*}
This new set $S^{\prime}$ of $n$ points will have the same induced directed graph and competition graph as the original set $S$. We may repeat this procedure if there are other $x$ coordinates that multiple points share. If there are points sharing a $y$ coordinate, the same approach may be applied by slightly shifting points ``downward'' to create a decreasing slope. This guarantees that we may always produce a set of points (with distinct $x$ and $y$ coordinates) having the same induced directed graph and competition graph as set $S$.
\end{proof}

The previous proposition shows that any doubly partial order for a finite subset of $\RR^{2}$ can be thought of as a doubly partial order on a permutation (in terms of the induced digraph and competition graph). In addition, we get some immediate corollaries from the results of Cho and Kim \cite{ChoKim}.

\begin{theorem}[Cho, Kim]
	Competition graphs of a digraph of a doubly partial order are interval graphs.
\end{theorem}

\begin{corollary}
	For each $\pi \in \Sn$, $C(\pi)$ is an interval graph.
\end{corollary}

\begin{theorem}[Cho, Kim]
	An interval graph with sufficiently many isolated vertices is the competition graph of a doubly partial order.
\end{theorem}

\begin{corollary}
	An interval graph with sufficiently many isolated vertices is the competition graph $C(\pi)$ of a doubly partial order of some permutation $\pi$.
\end{corollary}

We also make some observations on how structure within a permutation leads to structure within the competition graph. The following observation follows directly from the definition of the doubly partial order on permutations.

\begin{observation}
	Given a permutation $\pi = \pi_{1} \ldots \pi_{n}$, there is an arc $(\pi_{j}, \pi_{i})$ in digraph $D(\pi)$ if and only if $\pi_{i}$ and $\pi_{j}$ form a $1 2$ pattern (that is, $i < j$ and $\pi_{i} < \pi_{j}$).
\end{observation}

\noindent Edges in the competition graph also have a nice relation to structure within the permutation.

\begin{proposition}\label{propCGedge}
	Given a permutation $\pi = \pi_{1} \ldots \pi_{n}$, there is an edge $\{ \pi_{i}, \pi_{j} \}$ in the competition graph $C(\pi)$ if and only if $\pi_{i}$ and $\pi_{j}$ are the ``$2$'' and ``$3$'' terms in either a $1 2 3$ pattern or a $1 3 2$ pattern.
\end{proposition}

\begin{proof}
	The edge $\{ \pi_{i}, \pi_{j} \}$ is in the competition graph $C(\pi)$ if and only if there exists a $\pi_{k}$ such that $k < i, j$ and $\pi_{k} < \pi_{i}, \pi_{j}$, in which case the subsequence $\pi_{k} \pi_{i} \pi_{j}$ forms either a $1 2 3$ or a $1 3 2$ pattern in $\pi$.
\end{proof}

Also observe that if we considered the competition graph $C(\pi)$ with multiple edges allowed, there is a one-to-one correspondence between the edges in the graph and occurrences of $1 2 3$ and $1 3 2$ patterns in $\pi$. Such graphs will be considered in Section~\ref{wcompgraphs}.

\FloatBarrier
\section{Competition graphs of permutations}\label{compgraphs}

In the previous section, we showed that the set of directed graphs and competition graphs induced by permutations (i.e., $D(\Sn)$ and $C(\Sn)$) is equivalent to the set of such graphs induced by the doubly partial order on $\RR^{2}$ studied by Cho and Kim. Since edges in the competition graph correspond strictly to $123$ and $132$ patterns in the permutation, it is natural to consider restricting permutations to those that avoid one of the patterns. In this section, we will study competition graphs arising from permutations in $\Sn(123)$ and $\Sn(132)$. By Proposition~\ref{propCGedge}, edges in graphs from $C(\Sn(123))$ (resp.~$C(\Sn(132))$) correspond to occurrences of the pattern $1 3 2$ (resp.~pattern $1 2 3$). It should be emphasized that we will not consider multiple edges in the competition graphs in this section. The main result of this section is a forbidden subgraph characterization of graphs in $C(\Sn(132))$ along with some related observations for the $C(\Sn(123))$ case.

\subsection{Competition graphs of $\Sn(123)$ and $\Sn(132)$}

We first observe that for each $n$, there is a competition graph in $C(\Sn(123))$ and $C(\Sn(132))$ that contains $K_{n-1}$, the complete graph on $n-1$ vertices, as a subgraph.

\begin{proposition}
	Let the $n$-vertex graph $H$ be a copy of $K_{n-1}$ together with an isolated vertex. Then, $H$ is isomorphic to a graph in $C(\Sn(123))$ as well as a graph in $C(\Sn(132))$.
\end{proposition}

\begin{proof}
	The permutations $1 2 \ldots n \in \Sn(132)$ and $1 n (n-1) (n-2) \ldots 2 \in \Sn(123)$ produce the desired competition graphs.
\end{proof}

Recall that $s_{n}(123) = s_{n}(132)$ for all $n$. Even though $\Sn(123)$ and $\Sn(132)$ are in one-to-one correspondence (and multiple bijections are known), the sets of competition graphs $C(\Sn(123))$ and $C(\Sn(132))$ are not ``equivalent.'' For $n \leq 6$, it can be computationally verified that there exists a bijection from $C(\Sn(123))$ to $C(\Sn(132))$ mapping graphs in the first set to isomorphic graphs in the second set. For $n \geq 7$, there are graphs in $C(\Sn(123))$ (resp. $C(\Sn(132))$) that are not isomorphic to any graphs in $C(\Sn(132))$ (resp. $C(\Sn(123))$). 

The two graph structures of interest are stars and paths. Let $K_{1,m}$ denote the complete bipartite graph on $m+1$ vertices with bipartitions of size $1$ and $m$. We will often refer to this as a star. Also, let $P_{m}$ denote the path graph on $m+1$ vertices.

\begin{lemma}\label{lembasecase7}
	Let $K^{\prime}_{1,3}$ be the graph with $7$ vertices formed by a copy of $K_{1,3}$ and $3$ isolated vertices, and let $P^{\prime}_{3}$ be the graph on $7$ vertices formed by a copy of $P_{3}$ and $3$ isolated vertices.
	\begin{enumerate}
		\item[(i)] There exists no $\pi \in \mathcal{S}_{7}(123)$ such that $C(\pi) \cong K^{\prime}_{1,3}$. Additionally, $K^{\prime}_{1,3}$ is the only graph in $C(\mathcal{S}_{7}) \backslash C(\mathcal{S}_{7}(123))$.
		\item[(ii)] There exists no $\pi \in \mathcal{S}_{7}(132)$ such that $C(\pi) \cong P^{\prime}_{3}$. Additionally, $P^{\prime}_{3}$ is the only graph in $C(\mathcal{S}_{7}) \backslash C(\mathcal{S}_{7}(132))$.
	\end{enumerate}
\end{lemma}

\begin{proof}
	Both of these can be verified via computer. 
\end{proof}

A stronger statement can be shown for larger $n$.

\begin{proposition}
	\label{propindsg}
	Given a graph $G \in C(\Sn)$ with $n \geq 7$, the following hold:
	\begin{enumerate}
		\item[(i)] If $K_{1,3}$ is an induced subgraph of $G$, then $G \notin C(\Sn(123))$.
		\item[(ii)] If $P_{3}$ is an induced subgraph of $G$, then $G \notin C(\Sn(132))$.
	\end{enumerate}
\end{proposition}

\begin{proof}
	\begin{enumerate}
		\item[(i)] For the sake of contradiction, suppose $\pi \in \Sn(123)$ has a competition graph $G=C(\pi)$ that contains $K_{1,3}$ as an induced subgraph. Let $\pi_{a}$, $\pi_{b}$, $\pi_{c}$, and $\pi_{d}$ be the four vertices that form the $K_{1,3}$ induced subgraph. Observe that no single prey vertex can induce more than one edge in the $K_{1,3}$. To see this, suppose the prey vertex $\pi_{\alpha}$ induced two of the edges in the $K_{1,3}$. Without loss of generality, suppose those two edges are $\{ \pi_{a}, \pi_{b} \}$ and $\{ \pi_{a}, \pi_{c} \}$. Since $\alpha < a,b,c$ and $\pi_{\alpha} < \pi_{a}, \pi_{b}, \pi_{c}$, the induced subgraph should contain a triangle, which is a contradiction. Thus, each of the three edges has a distinct prey vertex (if an edge has more than one possible prey vertex, just choose one). This gives us $7$ vertices ($4$ predators and $3$ prey) and the corresponding length $7$ subsequence within $\pi$ would give us a length $7$ permutation $\sigma$ that produces $K_{1,3}$ with $3$ isolated vertices. If $\pi$ was $123$-avoiding, then $\sigma$ would also be $123$-avoiding, but by Lemma~\ref{lembasecase7}, this is impossible. Therefore, $\pi \notin \Sn(123)$ and $G \notin C(\Sn(123))$.
		\item[(ii)] The same argument as above holds by replacing $K_{1,3}$ with $P_{3}$ and $123$ with $132$.
	\end{enumerate}
\end{proof}

This leads to our main result of this section.

\begin{theorem}\label{thm132class}
	Let graph $G \in C(\Sn)$. Then, $G \in C(\Sn(132))$ if and only if $P_{3}$ is not an induced subgraph in $G$.
\end{theorem}

The proof of the theorem is given in the next subsection. We also have the analogous conjecture (with one direction proven by Proposition~\ref{propindsg}):

\begin{conjecture}
	Let graph $G \in C(\Sn)$. Then, $G \in C(\Sn(123))$ if and only if $K_{1,3}$ is not an induced subgraph in $G$.
\end{conjecture}

This conjecture would be resolved if one proves: ``if $K_{1,3}$ is not an induced subgraph in $G$, then $G \in C(\Sn(123))$.''

\FloatBarrier
\subsection{Proof of Theorem~\ref{thm132class}}

We first define an operation for inflating a permutation with smaller permutations. Let $\pi \in \Sn$ and $\sigma^{1}, \sigma^{2}, \ldots, \sigma^{n}$ be non-empty permutations (possibly of different lengths). The \emph{inflation} of $\pi$ by $\sigma^{1}, \ldots, \sigma^{n}$, denoted by $\pi [\sigma^{1}, \sigma^{2}, \ldots, \sigma^{n}]$, is the permutation of length $\sum | \sigma^{i} |$ created by replacing $\pi_{i}$ with $\sigma^{i}$. For example, $213[132, 12, 4231] = 354129786$, as shown in Figure~\ref{figinflation}.

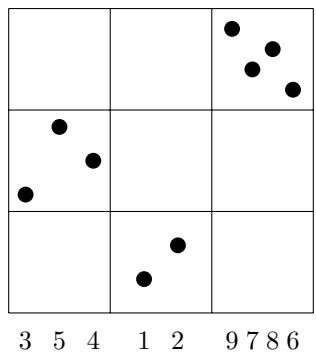
\begin{figure}[h!]
	\[\begin{tikzpicture}[scale=0.45]
		\node (p1) at (0.5,3.5) [circle,fill=black,inner sep=0pt,minimum size=6pt] {};
		\node (p2) at (1.5,5.5) [circle,fill=black,inner sep=0pt,minimum size=6pt] {};
		\node (p3) at (2.5,4.5) [circle,fill=black,inner sep=0pt,minimum size=6pt] {};
		\node (p4) at (4,1) [circle,fill=black,inner sep=0pt,minimum size=6pt] {};
		\node (p5) at (5,2) [circle,fill=black,inner sep=0pt,minimum size=6pt] {};
		\node (p6) at (6.6,8.4) [circle,fill=black,inner sep=0pt,minimum size=6pt] {};
		\node (p7) at (7.2,7.2) [circle,fill=black,inner sep=0pt,minimum size=6pt] {};
		\node (p8) at (7.8,7.8) [circle,fill=black,inner sep=0pt,minimum size=6pt] {};
		\node (p9) at (8.4,6.6) [circle,fill=black,inner sep=0pt,minimum size=6pt] {};
	
		\draw[step=3.0,black,thin] (0,0) grid (9,9);
		
		\node (lab1) at (0.5,0) [label=below:$3$] {};
		\node (lab2) at (1.5,0) [label=below:$5$] {};
		\node (lab3) at (2.5,0) [label=below:$4$] {};
		\node (lab4) at (4,0) [label=below:$1$] {};
		\node (lab5) at (5,0) [label=below:$2$] {};
		\node (lab6) at (6.6,0) [label=below:$9$] {};
		\node (lab7) at (7.2,0) [label=below:$7$] {};
		\node (lab8) at (7.8,0) [label=below:$8$] {};
		\node (lab9) at (8.4,0) [label=below:$6$] {};
	
	\end{tikzpicture}\]
	\caption{Graphical representation of $213[132, 12, 4231] = 354129786$}\label{figinflation}
\end{figure}

We begin by proving a couple of lemmas before the main result.

\begin{lemma}\label{lemconncomp}
	Let $\pi \in \Sn$ such that $G := C(\pi)$ has exactly one connected component $C_{1}$ with at least two vertices and $G$ avoids $P_{3}$ as an induced subgraph. Then, there exists a vertex $v \in C_{1}$ that is adjacent to all vertices in $C_{1} - v$.
\end{lemma}

\begin{proof}
	Given such a permutation $\pi = \pi_{1} \ldots \pi_{n}$, we will say that a term $\pi_{i}$ is \emph{redundant} in $\pi$ if the permutation $\sigma := \red(\pi - \pi_{i})$ has competition graph $C(\sigma)$ that is isomorphic to $C(\pi)$ with one isolated vertex deleted (i.e., it maintains the connected component). We construct a ``minimal'' permutation $\pi^{\prime}$ iteratively from $\pi$ as follows. Let $\sigma^{(0)} := \pi$. Now, if $\sigma^{(i)}$ contains no redundant terms, then $\pi^{\prime} := \sigma^{(i)}$. Otherwise, we define $\sigma^{(i)}_{\ast}$ as the right-most redundant term in $\sigma^{(i)}$ and define $\sigma^{(i+1)} := \red(\sigma^{(i)} - \sigma^{(i)}_{\ast})$. We note that the definition of $\pi^{\prime}$ is well-defined and any term in $\pi^{\prime}$ can be associated to its corresponding term in $\pi$. We also observe that the isolated vertices in $C(\pi^{\prime})$ will form a decreasing pattern in $\pi^{\prime}$ (if there were a $1 2$ pattern among those terms, the `$2$' term would be redundant).
	
	Let $m := |\pi^{\prime}|$ and let $\pi^{\prime}_{k}$ be the right-most term in $\pi^{\prime}$ that is an isolated vertex in $C(\pi^{\prime})$. We note that this term must induce at least one edge (and more precisely, a clique on the vertices $\pi^{\prime}_{k+1}, \ldots, \pi^{\prime}_{m}$). Let $\pi^{\prime}_{c}$ and $\pi^{\prime}_{d}$ be the largest and smallest terms out of $\pi^{\prime}_{k+1}, \ldots, \pi^{\prime}_{m}$, respectively. Let $\pi^{\prime}_{j}$ be the left-most term such that $\pi^{\prime}_{j}$ is an isolated vertex in $C(\pi^{\prime})$ and $\pi^{\prime}_{j} < \pi^{\prime}_{c}$. Note that $\pi^{\prime}_{j} > \pi^{\prime}_{d}$, otherwise $\pi^{\prime}_{k}$ is redundant.
	
	We claim that $\pi^{\prime}_{c}$ is our desired vertex. If $\pi^{\prime}_{j}$ is the left-most term in $\pi^{\prime}$ that is an isolated vertex, we are done. If $\pi^{\prime}_{j} = \pi^{\prime}_{k}$, then either it is the only isolated vertex (and we are again done) or there is another $\pi^{\prime}_{i}$ to the left of $\pi^{\prime}_{k}$ such that $\pi^{\prime}_{i} > \pi^{\prime}_{c}$, a contradiction (there would be more than one connected component).
	
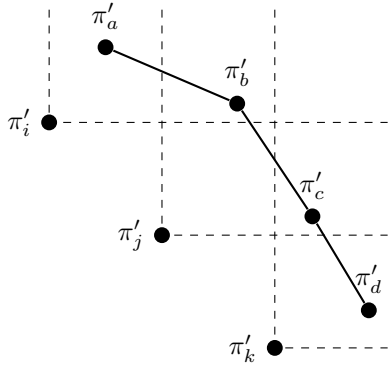
\begin{figure}[h!]
	\[\begin{tikzpicture}[scale=0.5]
		\node (p1) at (0,6) [circle,fill=black,label=left:$\pi^{\prime}_{i}$,inner sep=0pt,minimum size=6pt] {};
		\node (p2) at (3,3) [circle,fill=black,label=left:$\pi^{\prime}_{j}$,inner sep=0pt,minimum size=6pt] {};
		\node (p3) at (6,0) [circle,fill=black,label=left:$\pi^{\prime}_{k}$,inner sep=0pt,minimum size=6pt] {};
		\node (pa) at (1.5,8) [circle,fill=black,label=above:$\pi^{\prime}_{a}$,inner sep=0pt,minimum size=6pt] {};
		\node (pb) at (5,6.5) [circle,fill=black,label=above:$\pi^{\prime}_{b}$,inner sep=0pt,minimum size=6pt] {};
		\node (pc) at (7,3.5) [circle,fill=black,label=above:$\pi^{\prime}_{c}$,inner sep=0pt,minimum size=6pt] {};
		\node (pd) at (8.5,1) [circle,fill=black,label=above:$\pi^{\prime}_{d}$,inner sep=0pt,minimum size=6pt] {};
		
		\draw[dashed] (0,9)--(0,6)--(9,6);
		\draw[dashed] (3,9)--(3,3)--(9,3);
		\draw[dashed] (6,9)--(6,0)--(9,0);
		
		\draw[thick] (pa)--(pb)--(pc)--(pd);
		
	\end{tikzpicture}\]
	\caption{Structure within $\pi^{\prime}$ for Lemma~\ref{lemconncomp}}\label{figlemmaC1}
\end{figure}
	
	Now, suppose that $\pi^{\prime}_{j} \neq \pi^{\prime}_{k}$ and there is another isolated vertex to the left of $\pi^{\prime}_{j}$. Let $\pi^{\prime}_{b}$ be the largest term between $\pi^{\prime}_{j}$ and $\pi^{\prime}_{k}$. Also, let $\pi^{\prime}_{i}$ be the left-most term such that $\pi^{\prime}_{i}$ is an isolated vertex and $\pi^{\prime}_{i} < \pi^{\prime}_{b}$. Note that such a term must exist and must be to the left of $\pi^{\prime}_{j}$. Let $\pi^{\prime}_{a}$ be the largest term between $\pi^{\prime}_{i}$ and $\pi^{\prime}_{j}$. The terms $\pi^{\prime}_{a}, \pi^{\prime}_{b}, \pi^{\prime}_{c}, \pi^{\prime}_{d}$ now form an induced $P_{3}$ in $C(\pi^{\prime})$, a contradiction. Therefore, there is no $\pi^{\prime}_{i}$ to the left of $\pi^{\prime}_{j}$, so $\pi^{\prime}_{c}$ is adjacent to all vertices in its connected component.
\end{proof}

\begin{lemma}\label{lempartition}
	Let $\pi \in \Sn$, and suppose that $G := C(\pi)$ has the connected components $C_{1}, C_{2}, \ldots, C_{k}$ where $|C_{i}| > 1$ for each $i$. Then, $\pi$ can be partitioned into a collection of disjoint subsequences $L^{0}, L^{1}, \ldots, L^{k}$, with $\sigma^{i} := \red(L^{i})$ (for $0 \leq i \leq k)$, such that the following are true:
	\begin{enumerate}
		\item[(i)] $| \sigma^{0} | \geq 0$, $| \sigma^{i} | \geq 1$ (for $1 \leq i \leq k$), and $\sum |\sigma^{i}| = n$.
		\item[(ii)] $C(\sigma^{0})$ contains no edges.
		\item[(iii)] For $1 \leq i \leq k$, $C(\sigma^{i}) \cong C_{i}^{\prime}$, where $C_{i}^{\prime}$ is $C_{i}$ with $0$ or more isolated vertices.
		\item[(iv)] For each $i$ ($1 \leq i \leq k$), none of the isolated vertices in $C(\sigma^{i})$ form a $1 2$ pattern within $\sigma^{i}$.
	\end{enumerate}
\end{lemma}

\begin{proof}
	For each $C_{j}$, let $I_{j}$ be the set of (prey) vertices inducing edges in $C_{j}$. Note that $C_{j} \cap I_{j}$ may be non-empty. We first prove two claims.\\
	
	\textit{Claim 1: if $i \neq j$, then $I_{i} \cap I_{j} = \emptyset$}.
	
	Suppose there exists $x \in I_{i} \cap I_{j}$. Let $s,t \in C_{i}$ such that $x$ induces edge $s-t$ and $u,v \in C_{j}$ such that $x$ induces edge $u-v$. Then, the term $\pi_{x}$ (associated with $x$) must be to the left of and less than the terms $\pi_{s}, \pi_{t}, \pi_{u}, \pi_{v}$ (associated with $s,t,u,v$), so $x$ induces a $K_{4}$ subgraph on those four vertices. This contradicts $C_{i}$ and $C_{j}$ being different connected components.\\
	
	\textit{Claim 2: if $i \neq j$, then $C_{i} \cap I_{j} = \emptyset$}.
	
	Suppose there exists $x \in C_{i} \cap I_{j}$. Let $s,t \in C_{j}$ such that $x$ induces edge $s-t$. Then, the term in $\pi_{x}$ must be to the left of and less than the terms $\pi_{s}, \pi_{t}$. Also, let $y \in C_{i}$ be a vertex adjacent to $x$ and suppose that edge $x-y$ is induced by vertex $v \in I_{i}$. Then, the term $\pi_{v}$ is to the left of and less than the terms $\pi_{x}, \pi_{y}, \pi_{s}, \pi_{t}$. This would induce a $K_{4}$ on the four vertices, contradicting $C_{i}$ and $C_{j}$ being different connected components.\\
	
	Now, for each $1 \leq i \leq k$, we define the set $I_{i}^{\prime} \subseteq (I_{i} \backslash C_{i})$ as
\begin{equation*}
	I_{i}^{\prime} := \{ b \in I_{i} \backslash C_{i} \; : \; \text{there does not exist } a \in I_{i} \backslash C_{i} \text{ where } a b \text{ form a } 1 2 \text{ pattern} \}.
\end{equation*}
Let $L^{i}$ be the subsequence of terms in $\pi$ associated to the vertices $C_{i} \cup I_{i}^{\prime}$, and let $\sigma^{i} := \red(L^{i})$. Also, let $L^{0}$ be the (potentially empty) subsequence of terms in $\pi$ not contained in any $L^{i}$ for $i \geq 1$, and let $\sigma^{0} := \red(L^{0})$. Observe that this gives a partitioning of $\pi$ that satisfies conditions (i)--(iv).
\end{proof}

\noindent Observe that Theorem~\ref{thm132class} can be restated as follows:\\
\textit{``Let graph $G \in C(\Sn)$. Then, $G \notin C(\Sn(132))$ if and only if $P_{3}$ is an induced subgraph in $G$.''}\\

We now prove this equivalent statement.
\begin{proof}[Proof of Theorem~\ref{thm132class}]
	One direction is already given by Proposition~\ref{propindsg}. Now suppose $\pi \in \Sn$ such that $P_{3}$ is \emph{not} an induced subgraph in $G = C(\pi)$. We would like to show that there exists $\pi^{\prime} \in \Sn(132)$ such that $G \cong C(\pi^{\prime})$. Let $C_{1}, \ldots, C_{k}$ be the connected components of $G$ with at least two vertices. If there are no such connected components (i.e., $G$ has no edges), then $\pi^{\prime} = (n) (n-1) \ldots 2 1 \in \Sn(132)$ such that $C(\pi^{\prime}) \cong G$. For later notational convenience, we let $\alpha^{i}$ denote the decreasing permutation of length $i$. We now handle $k \geq 1$ in two cases.\\
	
	\textit{Case ($k=1$)}: Let $v \in C_{1}$ such that $v$ is adjacent to all vertices in $C_{1} - v$, given by Lemma~\ref{lemconncomp}, and let $\pi_{v}$ be the term in $\pi$ associated with $v$. Let $\beta := \red(\pi - \pi_{v})$ and $G^{\prime} := C(\beta)$. Observe that $G^{\prime} \cong G - v$ since $v$ is a non-isolated vertex. Let $Q_{1}, \ldots, Q_{j}$ be the connected components of $G^{\prime}$ with at least two vertices. Note that the $Q_{i}$ components (and perhaps some isolated vertices in $G^{\prime}$) are the ``pieces'' that were disconnected from $C_{1}$ when $v$ was deleted. Let $I$ be the set of new isolated vertices created from $C_{1} - v$, and let $m := |I|$. We note that $G$ must have had at least $m$ additional isolated vertices (prey) inducing the $m$ edges that were deleted to produce $I$. 
	
	Since $G^{\prime}$ also does not contain an induced $P_{3}$ subgraph, by induction, there exists $\pi^{\prime} \in \mathcal{S}_{n-1}(132)$ such that $C(\pi^{\prime}) \cong G^{\prime}$. By Lemma~\ref{lempartition}, $\pi^{\prime}$ can be partitioned into smaller permutations $\sigma^{0}, \sigma^{1}, \ldots, \sigma^{j}$ such that $C(\sigma^{0})$ contains no edges, $C(\sigma^{i}) \cong C_{i}$ (with some isolated vertices) for each $1 \leq i \leq j$, $C(\sigma^{i})$ does not contain any isolated vertices forming a $12$ pattern in $\sigma^{i}$ for each $1 \leq i \leq j$, and $C(\sigma^{0}) \cup C(\sigma^{1}) \cup \ldots \cup C(\sigma^{j}) \cong G^{\prime}$. Note that $|C(\sigma^{0})| \geq 2 m$. Observe that each $\sigma^{i}$ avoids $1 3 2$ since it is equivalent to a subsequence in $\pi^{\prime}$. Additionally, we replace $\sigma^{0}$ with $\tau = \alpha^{m}[1 2, 1 2, \ldots, 1 2]$ and the decreasing permutation of length $|C(\sigma^{0})| - 2 m$, which we call $\tau^{0}$. 
	
	First, observe that $\alpha^{j}[\sigma^{1}, \ldots, \sigma^{j}]$, $\tau$, and $\tau^{0}$ are all permutations avoiding $1 3 2$. Then, $\pi^{\prime \prime} := 4 2 1 3 [\tau^{0}, \alpha^{j}[\sigma^{1}, \ldots, \sigma^{j}], \tau, 1] \in \Sn(1 3 2)$, as shown in Figure~\ref{figthmpf1}. Additionally, $C(\pi^{\prime \prime}) \cong G$, as desired.\\

\begin{figure}[h!]
	\[\begin{tikzpicture}[scale=0.6]
		\draw (0,0) rectangle (8,7);
		\draw (0,6) rectangle (1,7);
		\draw (2,1) rectangle (6,5);
		\draw (2,4) rectangle (3,5);
		\draw (3,3) rectangle (4,4);
		\draw (5,1) rectangle (6,2);
		\draw (6,0) rectangle (7,1);
		\draw (7,5) rectangle (8,6);
	
		\node (p1) at (7.5,5.5) [circle,fill=black,inner sep=0pt,minimum size=5pt] {};
	
		\node (lab1) at (0.5,6.5) [label=center:$\tau^{0}$] {};
		\node (lab2) at (2.5,4.5) [label=center:$\sigma^{1}$] {};
		\node (lab3) at (3.5,3.5) [label=center:$\sigma^{2}$] {};
		\node (lab4) at (4.5,2.5) [label=center:$\ddots$] {};
		\node (lab5) at (5.5,1.5) [label=center:$\sigma^{j}$] {};
		\node (lab6) at (6.5,0.5) [label=center:$\tau$] {};
	
	\end{tikzpicture}\]
	\caption{Graphical representation of $\pi^{\prime \prime} := 4 2 1 3 [\tau^{0}, \alpha^{j}[\sigma^{1}, \ldots, \sigma^{j}], \tau, 1]$}\label{figthmpf1}
\end{figure}
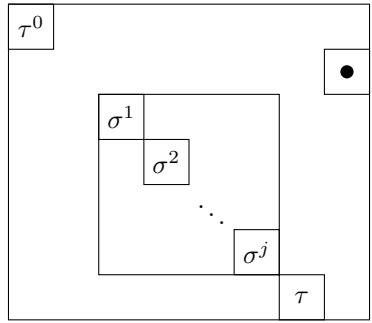
	
	\textit{Case ($k>1$)}: By Lemma~\ref{lempartition}, $\pi$ can be partitioned into smaller permutations $\sigma^{0}, \sigma^{1}, \ldots, \sigma^{k}$ such that $C(\sigma^{0})$ contains no edges, $C(\sigma^{i}) \cong C_{i}$ (with some isolated vertices) for each $1 \leq i \leq k$, and $C(\sigma^{0}) \cup C(\sigma^{1}) \cup \ldots \cup C(\sigma^{k}) \cong G$. Since each $C(\sigma^{i})$ avoids $P_{3}$ as an induced subgraph, by induction, there exists a $132$-avoiding permutation $\tau^{i}$ of the same length as $\sigma^{i}$ such that $C(\tau^{i}) \cong C(\sigma^{i})$ for each $i$. Then $\pi^{\prime} = \alpha^{k+1}[\tau^{0}, \tau^{1}, \ldots, \tau^{k}]$ avoids $132$ and $C(\pi) \cong C(\pi^{\prime})$. (See Figure~\ref{figthmpf2}.)
	
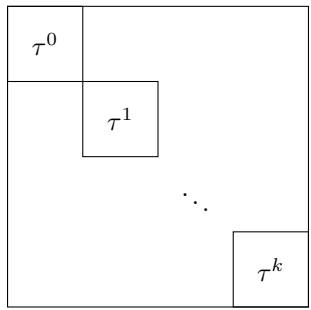
\begin{figure}[h!]
	\[\begin{tikzpicture}
		\draw (0,0) rectangle (4,4);
		\draw (0,3) rectangle (1,4);
		\draw (1,2) rectangle (2,3);
		\draw (3,0) rectangle (4,1);
	
		\node (lab1) at (0.5,3.5) [label=center:$\tau^{0}$] {};
		\node (lab2) at (1.5,2.5) [label=center:$\tau^{1}$] {};
		\node (lab3) at (2.5,1.5) [label=center:$\ddots$] {};
		\node (lab4) at (3.5,0.5) [label=center:$\tau^{k}$] {};
	
	\end{tikzpicture}\]
	\caption{Graphical representation of $\pi^{\prime} = \alpha^{k+1}[\tau^{0}, \tau^{1}, \ldots, \tau^{k}]$}\label{figthmpf2}
\end{figure}
	
\end{proof}

\FloatBarrier
\section{Weighted competition graphs of permutations}\label{wcompgraphs}

In this section, we study a generalization of competition graphs which was first introduced by Sano \cite{sano}. Given a directed graph $D = (V,A)$, we define the \emph{weighted competition graph}, denoted by $W(D)$, to be the edge-weighted undirected graph $(G,w)$ such that $G = (V,E)$ is the competition graph of $D$ and the weight for edge $uv$, denoted by $w(u v)$, is the number of vertices $x$ such that both arcs $(u,x)$ and $(v,x)$ are in $D$ (i.e., the number of common prey for $u$ and $v$).

Given an edge-weighted graph $G$, we define the sets $W_{n}^{-1}(G)$ and $W_{n}^{-1}(G; \; \tau)$ as:
\begin{align*}
	W_{n}^{-1}(G) &:= \{ \pi \in \Sn \; : \; W(\pi) \text{ is isomorphic to } G \text{ with sufficiently many isolated vertices} \}.\\
	W_{n}^{-1}(G; \; \tau) &:= \{ \pi \in \Sn(\tau) \; : \; W(\pi) \text{ is isomorphic to } G \text{ with sufficiently many isolated vertices} \}.
\end{align*}
We will still focus on the path graphs $P_{m}$ and star graphs $K_{1,m}$ since these are the natural graph structures of interest for graphs arising from $123$ and $132$ avoiding permutations. In this section, references to $P_{m}$ (resp. $K_{1,m}$) will refer to the path graphs (resp. star graphs) with edge weights all equal to $1$.

\subsection{Permutations producing paths and stars}

We begin our study of weighted competition graphs by considering permutations that produce paths and stars. We will show a natural bijection between $W^{-1}_{n}(P_{m})$ (permutations forming path graphs) and $W^{-1}_{n}(K_{1,m})$ (permutations forming star graphs). 

We first define some notation. Let $P_{m}$ be the path graph on the vertices $\{ p_{0}, p_{1}, \ldots, p_{m} \}$, where $p_{i}$ and $p_{i+1}$ are adjacent for $0 \leq i \leq m-1$. For $\pi \in W_{n}^{-1}(P_{m})$, we will often say ``$p_{i}$ in $\pi$'' to refer to the term in $\pi$ that $p_{i}$ corresponds to.\\

\begin{figure}[h!]
	\[\begin{tikzpicture}[scale=0.8]
		\tikzset{dashedge/.style={dashed}}

		\vertex (P0) at (0,0) [label=above:$p_{0}$] {};
		\vertex (P1) at (2,0) [label=above:$p_{1}$] {};
		\vertex (P2) at (4,0) [label=above:$p_{2}$] {};
		\vertex (Pm1) at (8,0) [label=above:$p_{m-1}$] {};
		\vertex (Pm) at (10,0) [label=above:$p_{m}$] {};
	
	  \node (e1) at (1,0) [label=above:$1$] {};
		\node (e2) at (3,0) [label=above:$1$] {};
		\node (em) at (9,0) [label=above:$1$] {};
		
	\path
		
		(P0) edge (P1)
		(P1) edge (P2)
		(P2) edge [dashedge] (Pm1)
		(Pm1) edge (Pm)
		;   
	\end{tikzpicture}\]
	
	\caption{Path graph $P_{m}$}
\end{figure}
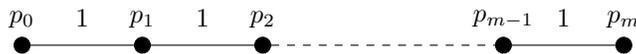

The following lemma provides insight into the arrangement of $p_{0}, \ldots, p_{m}$ within a permutation $\pi \in W_{n}^{-1}(P_{m})$.

\begin{lemma}
	\label{lemPathStruct}
	Let $\pi \in W_{n}^{-1}(P_{m})$. The $p_{i}$ labels on terms of $\pi$ can be assigned such that the following are true:
	\begin{enumerate}
		\item[(i)] $p_{i}$ is to the left of $p_{i+1}$ within $\pi$ for $0 \leq i \leq m-2$.
		\item[(ii)] $p_{m}$ is to the right of $p_{m-2}$ within $\pi$.
		\item[(iii)] If $\pi_{i_{1}}$ and $\pi_{i_{m}}$ (in $\pi$) correspond to $p_{1}$ and $p_{m}$ (in $P_{m}$), respectively, then $\pi_{i_{m}} < \pi_{i_{1}}$.
	\end{enumerate}
\end{lemma}

\begin{proof}
	\emph{Remark:} in the proof, we assume that $m \geq 4$. Cases with smaller $m$ are easily checked by computer.
	
	\textit{Satisfying condition (i):} We first prove that $p_{1}, p_{2}, \ldots, p_{m-1}$ must occur consecutively (in order) within $\pi$. Suppose that there is a $p_{k}$ that occurs between $p_{i}$ and $p_{i+1}$ within $\pi$, where none of the three vertices are $p_{0}$ or $p_{m}$. Let $x$ be the prey vertex inducing edge $p_{i} p_{i+1}$ and note that this must be to the left of and smaller in value than both vertices. Observe that $p_{k}$ must be smaller in value than $x$, otherwise $x$ would induce a triangle on the three vertices. Since $p_{k}$ has degree $2$, it has a neighbor $p_{k}^{\prime}$ which is neither $p_{i}$ nor $p_{i+1}$ as well as a prey vertex $y$ inducing the edge $p_{k} p_{k}^{\prime}$. Observe that $y$ must be to the left of $p_{k}$. But this induces a triangle on $p_{k}$, $p_{k}^{\prime}$, and whichever of $p_{i}$ or $p_{i+1}$ is to the right, a contradiction. 
	
	Therefore, the $p_{i}$'s in $\pi$ can be ordered either $p_{1} p_{2} \ldots p_{m-1}$ or $p_{m-1} p_{m-2} \ldots p_{1}$. We assume the former ordering and show that this satisfies all the other conditions. Now suppose that $p_{0}$ is to the right of $p_{1}$ in $\pi$. Let $y_{1,2}$ be the prey vertex inducing edge $p_{1} p_{2}$. The value of $p_{0}$ must be less than the value of $y_{1,2}$ in $\pi$, otherwise $y_{1,2}$ induces a triangle. Let $y_{0,1}$ be the prey vertex inducing edge $p_{0} p_{1}$. Observe that this vertex would induce a triangle on $p_{0}$, $p_{1}$, and $p_{2}$, a contradiction. Therefore, $p_{0}$ is to the left of $p_{1}$.
	
	\textit{Satisfying condition (ii):} Let prey vertices $a$, $b$, and $c$ induce edges $p_{m-3} p_{m-2}$, $p_{m-2} p_{m-1}$, and $p_{m-1} p_{m}$, respectively. Note that $a, b, c$ are distinct, otherwise they would induce a triangle. Let $\sigma$ be the length $7$ permutation order-isomorphic to the subsequence in $\pi$ corresponding to the terms $p_{m-3}$, $p_{m-2}$, $p_{m-1}$, $p_{m}$, $a$, $b$, and $c$. Note that $W(\sigma)$ should be isomorphic to $P_{3}$ with three isolated vertices. It is straightforward to computational check that the only length $7$ permutations forming this graph structure are $5736124$, $5736142$, $5637124$, and $5637142$ but none of these produce an instance where the ``$p_{m}$'' term is to the left of the ``$p_{m-2}$'' term. Therefore, $p_{m}$ is to the right of $p_{m-2}$.

	\textit{Satisfying condition (iii):} Note that $p_{m}$ is to the right of $p_{1}$ in $\pi$. If $p_{m}$'s value in $\pi$ was greater than $p_{1}$'s value in $\pi$, the prey vertex inducing edge $p_{0} p_{1}$ would also induce a triangle on $p_{0}$, $p_{1}$, and $p_{m}$, a contradiction.
\end{proof}	

The above labeling of $p_{i}$'s onto terms of $\pi$ will be considered the canonical labeling. We have a similar lemma regarding the structure of stars within a permutation. 

\begin{lemma}
	\label{lemStarStruct}
	Let $\pi \in W_{n}^{-1}(K_{1,m})$, and let $\{ a \}$ and $\{ b_{1}, \ldots, b_{m} \}$ be the bipartitions of $K_{1,m}$ such that $b_{i}$ is to the left of $b_{i+1}$ within $\pi$ for $1 \leq i \leq m-1$. The following are true:
	\begin{enumerate}
		\item[(i)] $a$ is to the right of $b_{m-1}$ within $\pi$.
		\item[(ii)] If $\pi_{\alpha}$ and $\pi_{\beta}$ (in $\pi$) correspond to $a$ and $b_{m}$ (in $K_{1,m}$), respectively, then $\pi_{\beta} < \pi_{\alpha}$.
	\end{enumerate}
\end{lemma}

\begin{proof}
	\begin{enumerate}
		\item[(i)] Suppose there are at least two $b_{i}$'s to the right of $a$ within $\pi$. Let $x$ and $y$ be the prey vertices inducing edges $\{ a, b_{m-1} \}$ and $\{ a, b_{m} \}$, respectively. Without loss of generality, suppose the value of $x$ within $\pi$ is less than the value of $y$ within $\pi$. Since $x$ is to the left of $a, b_{m-1}, b_{m}$, it induces a triangle within the competition graph, which is a contradiction.
		\item[(ii)] Suppose that $a$ is less than $b_{m}$ as terms in $\pi$. Both these terms are to the right of $b_{m-1}$, and $a$ and $b_{m-1}$ are adjacent in the competition graph. The prey vertex that induces the edge $\{ a, b_{m-1} \}$ would actually induce a triangle on $a, b_{m-1}, b_{m}$, a contradiction.
	\end{enumerate}
\end{proof}

We now show that the sets $W_{n}^{-1}(P_{m})$ and $W_{n}^{-1}(K_{1,m})$ are equinumerous.

\begin{theorem}
	\label{thmPathStar}
	For $m,n \geq 1$, $|W_{n}^{-1}(P_{m})| = |W_{n}^{-1}(K_{1,m})|$.
\end{theorem}

\begin{proof}
	We will prove this by establishing a bijection from $W_{n}^{-1}(P_{m})$ to $W_{n}^{-1}(K_{1,m})$. First, we may assume that $m \geq 3$ (otherwise, $P_{m} \cong K_{1,m}$ so this is trivially true). Additionally, we may assume that $n \geq 7$, since it is easy to verify by computer that smaller $n$ will not produce $P_{3}$ or $K_{1,3}$. Now consider an arbitrary $\pi = \pi_{1} \ldots \pi_{n} \in W_{n}^{-1}(P_{m})$. Let $\pi_{i_{j}}$ (in $\pi$) correspond to vertex $p_{j}$ (in $P_{m}$). We define the operator $T$ on $\pi$ to swap $p_{1}$ with the $p_{i}$ immediately to its right within $\pi$. Observe that $T^{k}$ on $\pi$ will be a cyclic left shift of $p_{1}, \ldots, p_{k+1}$ within $\pi$. More precisely, $T^{k}(\pi)$ is the permutation formed by replacing $\pi_{i_{j}}$ with $\pi_{i_{j+1}}$ for $1 \leq j \leq k$, replacing $\pi_{i_{k+1}}$ with $\pi_{i_{1}}$, and keeping all other terms of $\pi$ fixed. We will show that $T^{m-2}$ is a bijection from $W_{n}^{-1}(P_{m})$ to $W_{n}^{-1}(K_{1,m})$.
	
	Let $\pi^{\prime} = T^{k}(\pi)$. We claim that $W_{n}(\pi^{\prime})$ is the graph shown below in Figure~\ref{figpstar}.
	
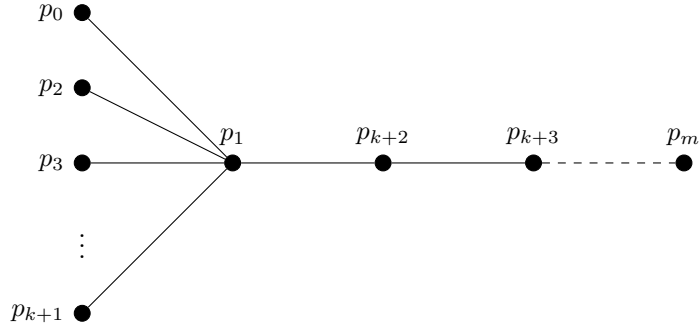
\begin{figure}[h!]
	\[\begin{tikzpicture}
		\tikzset{dashedge/.style={dashed}}

		\vertex (P0) at (0,4) [label=left:$p_{0}$] {};
		\vertex (P1) at (2,2) [label=above:$p_{1}$] {};
		\vertex (P2) at (0,3) [label=left:$p_{2}$] {};
		\vertex (P3) at (0,2) [label=left:$p_{3}$] {};
		\node (pdots) at (0,1) {$\vdots$};
		\vertex (Pk1) at (0,0) [label=left:$p_{k+1}$] {};
		\vertex (Pk2) at (4,2) [label=above:$p_{k+2}$] {};
		\vertex (Pk3) at (6,2) [label=above:$p_{k+3}$] {};
		\vertex (Pm) at (8,2) [label=above:$p_{m}$] {};
	
	\path
		
		(P0) edge (P1)
		(P2) edge (P1)
		(P3) edge (P1)
		(Pk1) edge (P1)
		(P1) edge (Pk2)
		(Pk2) edge (Pk3)
		(Pk3) edge [dashedge] (Pm)
		;   
	\end{tikzpicture}\]
	\caption{The graph $W_{n}(\pi^{\prime})$ (with isolated vertices not drawn).}\label{figpstar}
\end{figure}
	
	We prove the claim by induction on $k$. The base case $k = 0$ is trivially true since $\pi = T^{0}(\pi)$. Suppose the claim holds up to $k$, and let $\pi^{\prime} = \pi^{\prime}_{1} \ldots \pi^{\prime}_{n} := T^{k}(\pi)$. Also, let $G$ be the weighted competition graph $W(\pi^{\prime})$. Note that the operator $T$ applied to $\pi^{\prime}$ would swap the terms corresponding to $p_{1}$ and $p_{k+2}$ in $\pi^{\prime}$.
	
	Observe that $p_{1}$, $p_{k+2}$, and $p_{k+3}$ form a path of length $2$ in $G$. Let $x$ be the prey vertex in $D(\pi^{\prime})$ inducing the edge $\{ p_{1}, p_{k+2} \}$ in $G$, and let $y$ be the prey vertex inducing edge $\{ p_{k+2}, p_{k+3} \}$. Observe that $x \neq y$, otherwise that vertex would induce a triangle in $G$. Since all edge weights when edges exist are $1$, these are the unique such vertices. Let $\sigma$ be the length $5$ permutation that is order-isomorphic to the subsequence of $\pi$ corresponding to the terms $p_{1}$, $p_{k+2}$, $p_{k+3}$, $x$, and $y$. It is straightforward to computationally verify that $\sigma$ is either $3 4 1 5 2$ or $3 5 1 4 2$ if $k < m-2$. In both cases, the left to right order of the vertices are $x$, $p_{1}$, $y$, $p_{k+2}$, and $p_{k+3}$. By swapping the $p_{1}$ and $p_{k+2}$ terms, a direct verification on the two length $5$ permutations shows that: edge $\{ p_{1}, p_{k+2} \}$ is preserved, edge $\{ p_{k+2}, p_{k+3} \}$ is destroyed, and edge $\{ p_{1}, p_{k+3} \}$ is created.
	
	By Lemma~\ref{lemPathStruct}, all $p_{j}$ for $2 \leq j \leq k+1$ are to the left of $p_{1}, p_{k+2}, p_{k+3}$ in $\pi^{\prime}$ and swapping $p_{1}$ and $p_{k+2}$ will neither create new edges nor destroy existing ones that involve these $p_{j}$'s. Similarly, the remaining $p_{j}$'s will all be to the right of $p_{1}, p_{k+2}, p_{k+3}$ and the swap will neither create nor destroy any additional edges.
	
	For the final case $k = m-2$, we have two possibilities since $p_{m}$ can be to the right or the left of $p_{m-1}$, although both will be to the right of $p_{m-2}$ (by Lemma~\ref{lemPathStruct}). If $p_{m}$ is to the right of $p_{m-1}$, the argument above holds. If $p_{m}$ is to the left of $p_{m-1}$, then within $\pi^{\prime}$, we have the subsequence formed by $x, p_{1}, y, p_{m}, p_{m-1}$ and this subsequence is order-isomorphic to either $3 4 1 2 5$ or $3 5 1 2 4$. We swap $p_{1}$ and $p_{m-1}$, and checking the analogous cases as above completes the proof of our claim.
	
	The inverse mapping of $T^{-1}$ can also be defined and $(T^{-1})^{m-2}$ would send permutations $\pi = \pi_{1} \ldots \pi_{n}$ producing stars to permutations producing paths. Let $\{ a \}$ and $\{ b_{1}, b_{2}, \ldots, b_{m} \}$ be the bipartitions of $K_{1,m}$ such that $b_{i}$ is to the left of $b_{i+1}$ within $\pi$ for each $i$. By Lemma~\ref{lemStarStruct}, $a$ and $b_{m}$ are to the right of $b_{m-1}$ and $a$ is larger than $b_{m}$ as terms in $\pi$, so given a permutation $\pi \in W_{n}^{-1}(K_{1,m})$, we can distinguish which terms of $\pi$ correspond to each $b_{i}$ and $a$. The map $T^{-1}$ swaps $a$ with the $b_{i}$ immediately to $a$'s left within $\pi$. Comparing the $T$ and $T^{-1}$ maps, the $p_{1}$ term will correspond to $a$, the $p_{0}$ term will correspond to $b_{1}$, and the $p_{i}$ term will correspond to $b_{i}$ for $2 \leq i \leq m$.

\end{proof}

\FloatBarrier
\subsection{Pattern-avoiders producing paths and stars}

We now study the structure of permutations in $W^{-1}_{n}(K_{1,m}; \; 132)$. Additionally, we will show a bijection between $W^{-1}_{n}(P_{m}; \; 123)$ and $W^{-1}_{n}(K_{1,m}; \; 132)$, which will be very similar to the one in Theorem~\ref{thmPathStar}.

We first define the notion of an ``accessory vertex.'' Given $\pi \in \Sn$, we say that $\pi_{i}$ is an \emph{accessory term} if it does not contribute to any edges (as either an edge endpoint or a prey vertex) in the weighted competition graph $W(\pi)$. An isolated vertex in $W(\pi)$ corresponding to an accessory term will be called an \emph{accessory vertex}. Note that equivalently, an accessory term in $\pi$ is any term that is not part of a $123$ or $132$ pattern. Given a graph $G$, we say that a permutation in $W_{n}^{-1}(G)$ is a \emph{base permutation} if it contains no accessory terms. These permutations may be thought of as the minimal permutations that produce the desired graph.

Given a graph $G$, we define $B(G)$ to be the set of all base permutations in $\mathop{\bigcup} \limits_{n} {W_{n}^{-1}(G)}$. Given a permutation pattern $\tau$, we analogously define $B(G; \; \tau)$ to be the set of all base permutations in $\mathop{\bigcup} \limits_{n} {W_{n}^{-1}(G; \; \tau)}$. For example, $B(K_{1,3}) = \{ 5634127, 5634172, 5734126, 5734162 \}$ and $B(K_{1,3}; \; 132) = \{ 5634127 \}$.

\begin{proposition}
	For each $m$, $| B(K_{1,m}; \; 132) | = 1$ and that single base permutation is
	\begin{align*}
		\pi = (2 m - 1) (2 m) (2 m - 3) (2 m - 2) \ldots (3) (4) (1) (2) (2 m + 1).
	\end{align*}
\end{proposition}

\begin{proof}
	It is straightforward to verify that the permutation above is a base permutation producing $K_{1,m}$. To show that it is the only such permutation, we proceed by induction on $m$. It is easy to verify that $B(K_{1,1}; \; 132) = \{ 123 \}$. Let $\pi = \pi_{1} \ldots \pi_{n} \in B(K_{1,m}; \; 132)$ for $m \geq 2$. First observe that $\pi_{n-2} = 1$. If $1$ were to the left of $\pi_{n-2}$, it would induce a triangle on $\pi_{n-2}, \pi_{n-1}, \pi_{n}$, and if $1$ were to the right of $\pi_{n-2}$, it would be an accessory term. Similarly, note that $\pi_{n-1} = 2$. If $2$ were to the left of $\pi_{n-1}$, there would be multiple edges between $\pi_{n-1}$ and $\pi_{n}$, and if it were to the right (at the $\pi_{n}$ position), it would form a $132$ pattern. Then, $\pi^{\prime} = \red(\pi_{1} \ldots \pi_{n-3} \pi_{n}) \in B(K_{1,m-1}; \; 132)$, and since there is only one possible choice of $\pi^{\prime}$, there is only one possible choice for the original $\pi$.
\end{proof}

We will write $b(K_{1,m}; \; 132)$ to be the unique permutation in $B(K_{1,m}; \; 132)$. In the following results, we will let $Y$ denote an instance of a prey vertex and $P$ denote an instance of a predator vertex (within a permutation). By this notation, $b(K_{1,m}; \; 132) = (Y P)^{m} P = Y P Y P \ldots Y P P$. We will refer to the $i$-th $Y$ and the $i$-th $P$ in $b(K_{1,m}; \; 132)$ as the $i$-th $(Y P)$-pair (so $b(K_{1,m}; \; 132)$ is $m$ $(Y P)$-pairs followed by a $P$). We will refer to the last $P$ term as $P^{*}$.

\begin{lemma}
	Let $\pi \in W_{n}^{-1}(K_{1,m}; \; 132)$. The accessory terms within $\pi$ will never occur within a $(Y P)$-pair.
\end{lemma}

\begin{proof}
	Suppose on the contrary that accessory term $x$ occurs between the $i$-th $Y$ and the $i$-th $P$ terms (say $y_{i}$ and $p_{i}$). If $y_{i} < x$ (as terms in $\pi$), then $y_{i}$ would induce an edge on $x$ and $p_{i}$, and if $y_{i} > x$, then $x$ would induce an extra edge on $p_{i}$ and $P^{*}$. Both cases contradict $x$ being an accessory term.
\end{proof}

\begin{lemma}
	\label{lemYP2}
	Let $\pi \in W_{n}^{-1}(K_{1,m}; \; 132)$. For each $1 \leq i \leq m-1$, the accessory terms between the $i$-th $(Y P)$-pair and the $(i+1)$-th $(Y P)$-pair must be less than the $i$-th $(Y P)$-pair, greater than the $(i+1)$-th $(Y P)$-pair, and in decreasing order.
\end{lemma}

\begin{proof}
	Let accessory terms $a_{1}, a_{2}, \ldots, a_{k}$ be between the $i$-th $(Y P)$-pair (given by terms $y_{i}, p_{i}$) and the $(i+1)$-th $(Y P)$-pair (given by terms $y_{i+1}, p_{i+1}$) such that $a_{j}$ is to the left of $a_{j+1}$ for each $j$. Note that $y_{i} < p_{i}$ and $y_{i+1} < p_{i+1}$. If $y_{i} < a_{j}$, then $y_{i}$ induces an edge on $a_{j}$ and $P^{*}$, which contradicts $a_{j}$ being an accessory term. Therefore, each $a_{j}$ is less than the $i$-th $(Y P)$-pair.
	
	Similarly, if $a_{j} < p_{i+1}$, then $a_{j}$ induces a second edge on $p_{i+1}$ and $P^{*}$. Therefore, each $a_{j}$ is greater than the $i+1$-th $(Y P)$-pair. Finally, if there exists $1 \leq j < j^{\prime} \leq k$ such that $a_{j} < a_{j^{\prime}}$, then $a_{j}$ induces an edge on $a_{j^{\prime}}$ and $P^{*}$, which contradicts $a_{j^{\prime}}$ being an accessory term. Therefore, $a_{1} > a_{2} > \ldots > a_{k}$.
\end{proof}

Now, let $h(m,n) := |W_{n}^{-1}(K_{1,m}; 1 3 2)|$. We get the following nice recurrence:

\begin{theorem}
	\label{thmrecur}
	For $m > 1$ and $n > 2$, $h(m,n) = h(m,n-1) + h(m-1,n-2)$.
\end{theorem}

\begin{proof}
	First, for $n=1$ or $n=2$, $h(m,n) = 0$ if $m \geq 1$. Also, if $m=1$ and $n \geq 2$, then $h(m,n) = (n-2) 2^{n-3}$ by \cite{robertson}. Consider $\pi \in W_{n}^{-1}(K_{1,m}; 1 3 2)$ for $m > 1$ and $n > 2$. Recall that $\pi$ will have $m$ $(Y P)$-pairs. Suppose that the first $(Y P)$-pair in $\pi$ is given by $\pi_{a} \pi_{a+1}$. If $\pi_{a+2}$ is an accessory term, then $\pi^{\prime} = \red(\pi_{1} \ldots \pi_{a+1} \pi_{a+3} \ldots \pi_{n})$ is a permutation in $W_{n-1}^{-1}(K_{1,m}; 1 3 2)$, and since there is a unique way to re-insert $\pi_{a+2}$ (by Lemma~\ref{lemYP2}), the number of such original permutations $\pi$ is $| W_{n-1}^{-1}(K_{1,m}; 1 3 2) |$. On the other hand, if $\pi_{a+2}$ is part of the $2$nd $(Y P)$-pair, then $\pi^{\prime \prime} = \red(\pi_{1} \ldots \pi_{a+1} \pi_{a+4} \ldots \pi_{n})$ is a permutation in $W_{n-2}^{-1}(K_{1,m-1};  1 3 2)$, and there is a unique way to re-insert the $(Y P)$-pair given by $\pi_{a+2} \pi_{a+3}$, so the theorem follows.
\end{proof}

Finally, we prove the result analogous to Theorem~\ref{thmPathStar} also by a bijection.

\begin{theorem}
	\label{thmPathStar2}
	For $m,n \geq 1$, $|W_{n}^{-1}(P_{m}; 1 2 3)| = |W_{n}^{-1}(K_{1,m}; 1 3 2)|$.
\end{theorem}

\begin{proof}
	We prove this by establishing a bijection from $W_{n}^{-1}(P_{m}; 1 2 3)$ to $W_{n}^{-1}(K_{1,m}; 1 3 2)$. We may assume that $n \geq 3$ (otherwise, the competition graph cannot have any edges). Now consider an arbitrary $\pi = \pi_{1} \ldots \pi_{n} \in W_{n}^{-1}(P_{m}; 123)$, and let $\pi_{i_{j}}$ (in $\pi$) correspond to vertex $p_{j}$ (in $P_{m}$). First, observe that $\pi_{i_{0}} > \pi_{i_{1}} > \ldots > \pi_{i_{m}}$. If there exists $\pi_{i_{a}} < \pi_{i_{b}}$ with $a < b$, then these two together with a prey vertex of $\pi_{i_{a}}$ would form a $123$ pattern.
	
	We now define the map $M$ as a left cyclic shift on $p_{0}, p_{1}, \ldots, p_{m}$ while holding all other terms of $\pi$ fixed. More precisely, $M(\pi)$ is the permutation obtained by replacing $\pi_{i_{j}}$ with $\pi_{i_{j+1}}$ for $0 \leq j \leq m-1$ and replacing $\pi_{i_{m}}$ with $\pi_{i_{0}}$. An example of the mapping is given in Figure~\ref{figthmbijection}.\\
	
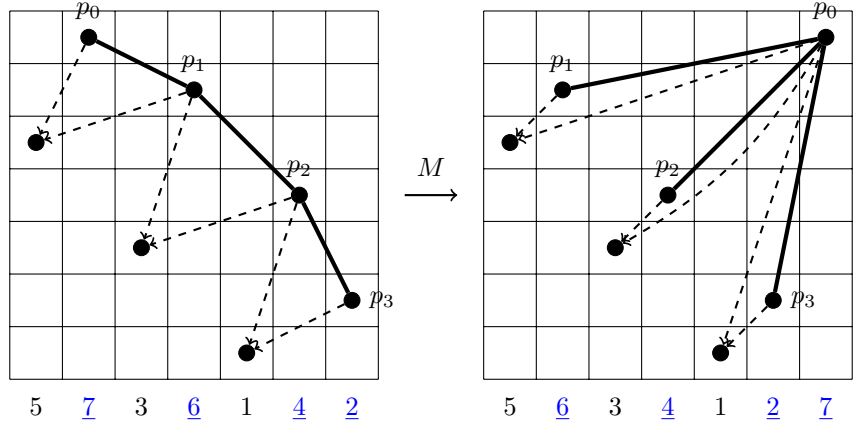
\begin{figure}[h!]
	\[\begin{tikzpicture}[scale=0.7]
		\draw[step=1.0,black,thin] (0,0) grid (7,7);
		\draw[step=1.0,black,thin] (9,0) grid (16,7);
	
		\vertex (p1) at (0.5,4.5) [circle,fill=black,inner sep=0pt,minimum size=6pt] {};
		\vertex (p2) at (1.5,6.5) [circle,fill=black,label=above:$p_{0}$,inner sep=0pt,minimum size=6pt] {};
		\vertex (p3) at (2.5,2.5) [circle,fill=black,inner sep=0pt,minimum size=6pt] {};
		\vertex (p4) at (3.5,5.5) [circle,fill=black,label=above:$p_{1}$,inner sep=0pt,minimum size=6pt] {};
		\vertex (p5) at (4.5,0.5) [circle,fill=black,inner sep=0pt,minimum size=6pt] {};
		\vertex (p6) at (5.5,3.5) [circle,fill=black,label=above:$p_{2}$,inner sep=0pt,minimum size=6pt] {};
		\vertex (p7) at (6.5,1.5) [circle,fill=black,label=right:$p_{3}$,inner sep=0pt,minimum size=6pt] {};
	
		\draw[ultra thick] (p2) -- (p4);
		\draw[ultra thick] (p4) -- (p6);
		\draw[ultra thick] (p6) -- (p7);
	
		\draw[->,dashed,thick] (p2) -- (p1);
		\draw[->,dashed,thick] (p4) -- (p1);
		\draw[->,dashed,thick] (p4) -- (p3);
		\draw[->,dashed,thick] (p6) -- (p3);
		\draw[->,dashed,thick] (p6) -- (p5);
		\draw[->,dashed,thick] (p7) -- (p5);
	
		\node (labp1) at (0.5,0) [label=below:$5$] {};
		\node (labp2) at (1.5,0) [label=below:\color{blue}$\underline{7}$] {};
		\node (labp3) at (2.5,0) [label=below:$3$] {};
		\node (labp4) at (3.5,0) [label=below:\color{blue}$\underline{6}$] {};
		\node (labp5) at (4.5,0) [label=below:$1$] {};
		\node (labp6) at (5.5,0) [label=below:\color{blue}$\underline{4}$] {};
		\node (labp7) at (6.5,0) [label=below:\color{blue}$\underline{2}$] {};
	
		\vertex (q1) at (9.5,4.5) [circle,fill=black,inner sep=0pt,minimum size=6pt] {};
		\vertex (q2) at (10.5,5.5) [circle,fill=black,label=above:$p_{1}$,inner sep=0pt,minimum size=6pt] {};
		\vertex (q3) at (11.5,2.5) [circle,fill=black,inner sep=0pt,minimum size=6pt] {};
		\vertex (q4) at (12.5,3.5) [circle,fill=black,label=above:$p_{2}$,inner sep=0pt,minimum size=6pt] {};
		\vertex (q5) at (13.5,0.5) [circle,fill=black,inner sep=0pt,minimum size=6pt] {};
		\vertex (q6) at (14.5,1.5) [circle,fill=black,label=right:$p_{3}$,inner sep=0pt,minimum size=6pt] {};
		\vertex (q7) at (15.5,6.5) [circle,fill=black,label=above:$p_{0}$,inner sep=0pt,minimum size=6pt] {};
	
		\draw[ultra thick] (q2) -- (q7);
		\draw[ultra thick] (q4) -- (q7);
		\draw[ultra thick] (q6) -- (q7);
	
		\draw[->,dashed,thick] (q2) -- (q1);
		\draw[->,dashed,thick] (q4) -- (q3);
		\draw[->,dashed,thick] (q6) -- (q5);
		\draw[->,dashed,thick] (q7) -- (q1);
		\draw[->,dashed,thick] (q7) to[bend left=15] (q3);
		\draw[->,dashed,thick] (q7) -- (q5);
		
		\node (labq1) at (9.5,0) [label=below:$5$] {};
		\node (labq2) at (10.5,0) [label=below:\color{blue}$\underline{6}$] {};
		\node (labq3) at (11.5,0) [label=below:$3$] {};
		\node (labq4) at (12.5,0) [label=below:\color{blue}$\underline{4}$] {};
		\node (labq5) at (13.5,0) [label=below:$1$] {};
		\node (labq6) at (14.5,0) [label=below:\color{blue}$\underline{2}$] {};
		\node (labq7) at (15.5,0) [label=below:\color{blue}$\underline{7}$] {};
	
		\draw[->,thick] (7.5,3.5) -- (8.5,3.5);
		\node (mapM) at (8,3.5) [label=above:$M$] {};
	
	\end{tikzpicture}\]
	\caption{Example of mapping $M$ on $5736142$}\label{figthmbijection}
\end{figure}
	
	First, we observe that map $M$ is equivalent to swapping $p_{0}$ and $p_{1}$ and then applying $T^{m-1}$, the operator defined in the proof of Theorem~\ref{thmPathStar} (while viewing the $p_{0}$ in $\pi$ to be the new ``$p_{1}$'' term that gets moved). Let $x$ and $y$ be the prey vertices inducing edges $p_{0} p_{1}$ and $p_{1} p_{2}$, respectively. Let $\sigma$ be the length $5$ permutation that is order-isomorphic to the subsequence of $\pi$ corresponding to $x, y, p_{0}, p_{1}, p_{2}$. It is straightforward to computationally verify that $ \sigma = 3 5 1 4 2$ is the only possibility. Swapping the $p_{0}$ and $p_{1}$ terms gives us $3 4 1 5 2$, which is another length two path (but this time, the original $p_{0}$ is now the middle vertex). Hence, swapping $p_{0}$ and $p_{1}$ in $\pi$ gives us a path where $p_{1}$ is now the endpoint and $p_{0}$ is its only neighbor. Using the same justification as the proof of Theorem~\ref{thmPathStar}, we may apply $T^{m-2}$ to this new permutation to get a permutation $\pi^{\prime}$ that produces a $K_{1,m}$. Given the original structure of $\pi$, it is clear that $W(\pi^{\prime})$ is isomorphic to $W(T(\pi^{\prime}))$.
	
	Also, due to the original structure of $\pi$, there are at least $m$ occurrences of $1 2 3$ in $T(\pi^{\prime})$, and since $W(T(\pi^{\prime}))$ has $m$ edges, there can be no occurrences of $1 3 2$. Since $M$ has a well-defined inverse (by performing a right cyclic shift on $p_{0}, p_{1}, \ldots, p_{m}$) and the previous results of this section show that all permutations in $W_{n}^{-1}(K_{1,m}; 1 3 2)$ have the structure of permutations arising from $M(\pi)$ with $\pi \in W_{n}^{-1}(P_{m}; 123)$, $M$ is a bijection between the sets.
\end{proof}

\FloatBarrier
\subsection{Some enumerative results}

Using the recurrence in Theorem~\ref{thmrecur}, we can quickly compute many values for the quantity $h(m,n)$.

\begin{table}[!h]
	\centering
  \begin{tabular}{|c|r|r|r|r|r|r|r|r|r|r|r|r|}
		\hline
		$m \backslash n$ & $1$ & $2$ & $3$ & $4$ & $5$ & $6$ & $7$ & $8$ & $9$ & $10$ & $11$ & $12$ \\ \hline
    $1$ & $0$ & $0$ & $1$ & $4$ & $12$ & $32$ & $80$ & $192$ & $448$ & $1024$ & $2304$ & $5120$ \\ \hline
    $2$ & $0$ & $0$ & $0$ & $0$ & $1$ & $5$ & $17$ & $49$ & $129$ & $321$ & $769$ & $1793$ \\ \hline
    $3$ & $0$ & $0$ & $0$ & $0$ & $0$ & $0$ & $1$ & $6$ & $23$ & $72$ & $201$ & $522$ \\ \hline
    $4$ & $0$ & $0$ & $0$ & $0$ & $0$ & $0$ & $0$ & $0$ & $1$ & $7$ & $30$ & $102$ \\ \hline
		$5$ & $0$ & $0$ & $0$ & $0$ & $0$ & $0$ & $0$ & $0$ & $0$ & $0$ & $1$ & $8$ \\ \hline
  \end{tabular}
	\caption{Some values for $h(m,n) = |W_{n}^{-1}(K_{1,m}; 1 3 2)|$.}
\end{table}

\noindent Recall that the first row is given by $h(1,n) = (n-2) 2^{n-3}$ for $n \geq 2$, since it is the number of $\pi \in \Sn(132)$ with exactly one copy of $123$.

We can also derive the generating function. For each $m$, we define 
\begin{align*}
	F_{m}(y) := \mathop{\sum} \limits_{n \geq 0} {h(m,n) y^{n}}.
\end{align*}
We also define the more general generating function
\begin{align*}
	H(x,y) := \mathop{\sum} \limits_{m,n \geq 0} {h(m,n) x^{m} y^{n}}.
\end{align*}
Note that the coefficient of $x^{m}$ in $H(x,y)$ is $F_{m}(y)$.

Since an explicit closed form is known for $h(1,n)$, it is straightforward to verify that $F_{1}(y) = y^{3}/(1 - 2 y)^{2}$. In addition, the recurrence $h(m,n) = h(m,n-1) + h(m-1,n-2)$ from Theorem~\ref{thmrecur} can be re-written as
\begin{align}
	h(m,n) = \mathop{\sum} \limits_{j=1}^{n-2} {h(m-1,j)}
\end{align}
so we get that $F_{m}(y) = F_{m-1}(y) \cdot (y^{2}/(1-y))$. Therefore, we get that
\begin{align}
	F_{m}(y) = \frac{y^{2 m + 1}}{(1 - 2 y)^{2} (1 - y)^{m-1}}.
\end{align}
Through some routine manipulations, we can also get the closed form for the bivariate rational generating function $H(x,y)$:
\begin{align}
	H(x,y) &= \mathop{\sum} \limits_{m=1}^{\infty} {F_{m}(y) x^{m}}\\
	&= \mathop{\sum} \limits_{m=1}^{\infty} {\frac{y^{2 m + 1}}{(1 - 2 y)^{2} (1 - y)^{m-1}} x^{m}}\\
	&= \frac{x y^{3} \left( 1 - y \right)}{\left( 1 - 2 y \right)^{2} \left(1 - y - x y^{2} \right)}.
\end{align}

Using the closed form expression for $F_{m}(y)$, we can also derive closed form expressions of $h(m,n)$ for fixed values of $m$. For example,
\begin{table}[!h]
	\centering
	\begin{tabular}{ll}
		$h(1,n) = (n-2) 2^{n-3}$ & , $n \geq 3$\\
		$h(2,n) = (n-5) 2^{n-4} + 1$ & , $n \geq 5$\\
		$h(3,n) = (n-8) 2^{n-5} + n - 2$ & , $n \geq 7$.
	\end{tabular}
\end{table}

We note that for smaller $m$ values, the associated sequences are known in the On-Line Encyclopedia of Integer Sequences \cite{OEIS}. For example, $h(1,n)$ is A001787, $h(2,n)$ is A000337, and $h(3,n)$ is A045618.

\FloatBarrier
\section{Conclusion}\label{concl}

In this article, we considered permutations inducing competition graphs through the notion of doubly partial orders. This led to interesting structural connections between permutations and competition graphs. Since edges in the competition graph arise precisely due to $123$ or $132$ patterns in the permutation, it was natural to consider permutations restricting one of the patterns. We were able to prove a classification for graphs arising from $132$-avoiding permutations and have a nice analogous conjecture for $123$-avoiding permutations. However, many other interesting questions remain. For example, do other patterns in permutations (longer patterns or other ``types'' of patterns) have any interesting connection to competition graphs? Do any permutation statistics carry over to competition graphs in any meaningful way?

We also considered the notion of weighted competition graphs, since this better captures some internal structure of the permutation. In particular, the edges in the graph are in one-to-one correspondence with $123$ and $132$ patterns in the permutation. We consider some enumerative and structural properties for permutations inducing certain weighted graphs (namely paths and stars). Many potential avenues of investigation also remain for weighted competition graphs, and perhaps related graphs, induced by permutations.\\

\noindent \textbf{Acknowledgments}: This work was initiated during the 2014 DIMACS REU program and the second author was supported by NSF grant CNS-1263082.

\bibliography{CompGraph}{}

\begin{thebibliography}{10}

\bibitem{ChoKim}
Han~Hyuk Cho and Suh-Ryung Kim.
\newblock A class of acyclic digraphs with interval competition graphs.
\newblock {\em Discrete Appl. Math.}, 148(2):171--180, 2005.

\bibitem{cohen}
Joel~E. Cohen.
\newblock Interval graphs and food webs: a finding and a problem.
\newblock {\em RAND Corporation Document 17696-PR}, 1968.

\bibitem{duttbrig}
R.~D. Dutton and R.~C. Brigham.
\newblock A characterization of competition graphs.
\newblock {\em Discrete Appl. Math.}, 6(3):315--317, 1983.

\bibitem{FLMMP}
Kathryn~F. Fraughnaugh, J.~Richard Lundgren, Sarah~K. Merz, John~S. Maybee, and
  Norman~J. Pullman.
\newblock Competition graphs of strongly connected and {H}amiltonian digraphs.
\newblock {\em SIAM J. Discrete Math.}, 8(2):179--185, 1995.

\bibitem{guichard}
David~R. Guichard.
\newblock Competition graphs of {H}amiltonian digraphs.
\newblock {\em SIAM J. Discrete Math.}, 11(1):128--134 (electronic), 1998.

\bibitem{kimrob}
Suh-Ryung Kim and Fred~S. Roberts.
\newblock Competition graphs of semiorders and the conditions {$C(p)$} and
  {$C^*(p)$}.
\newblock {\em Ars Combin.}, 63:161--173, 2002.

\bibitem{kitaev:book}
Sergey Kitaev.
\newblock {\em Patterns in permutations and words}.
\newblock Monographs in Theoretical Computer Science. An EATCS Series.
  Springer, Heidelberg, 2011.
\newblock With a foreword by Jeffrey B. Remmel.

\bibitem{kitman:survey}
Sergey Kitaev and Toufik Mansour.
\newblock A survey on certain pattern problems.
\newblock {\em University of Kentucky research report 2003�-09}, 2003.

\bibitem{knuth}
Donald~E. Knuth.
\newblock {\em The art of computer programming}.
\newblock Vol. 1: Fundamental algorithms. Addison Wesley, Reading,
  Massachusetts, 1973.

\bibitem{robste}
Fred~S. Roberts and Jeffrey~E. Steif.
\newblock A characterization of competition graphs of arbitrary digraphs.
\newblock {\em Discrete Appl. Math.}, 6(3):323--326, 1983.

\bibitem{robertson}
Aaron Robertson.
\newblock Permutations containing and avoiding {$123$} and {$132$} patterns.
\newblock {\em Discrete Math. Theor. Comput. Sci.}, 3(4):151--154 (electronic),
  1999.

\bibitem{sano}
Yoshio Sano.
\newblock Weighted competition graphs.
\newblock {\em preprint}, 2007.

\bibitem{OEIS}
Neil Sloane.
\newblock The {O}n-{L}ine {E}ncyclopedia of {I}nteger {S}equences,
  \url{http://oeis.org/}, 2015.

\bibitem{stein:survey}
Einar Steingr{\'{\i}}msson.
\newblock Some open problems on permutation patterns.
\newblock {\em London Mathematical Society Lecture Note Series, to appear}.

\end{thebibliography}
\bibliographystyle{plain}

\end{document}